\documentclass{commat}

\usepackage{amsmath, amsthm, amsfonts, epsfig}
\usepackage{amssymb, latexsym}
\usepackage{url}

 \newcommand{\mc}{\mathcal}
 \newcommand{\A}{\mc A}
 \newcommand{\B}{\mc B}
 \newcommand{\F}{\mc F}
 \newcommand{\G}{\mc G}

 \newcommand{\C}{\mathbb{C}}
 \newcommand{\R}{\mathbb{R}}
 \newcommand{\Z}{\mathbb{Z}}

 \newcommand{\ep}{\varepsilon}
 \newcommand{\bo}{\boldsymbol 0}

 \newcommand{\be}{\boldsymbol e}

 \newcommand{\bm}{\boldsymbol m}
 \newcommand{\bn}{\boldsymbol n}

 \newcommand{\bt}{\boldsymbol t}

 \newcommand{\bx}{\boldsymbol x}
 \newcommand{\bwy}{\boldsymbol y}

 \newcommand{\bz}{\boldsymbol z}
 \newcommand{\wF}{\widehat{F}}
 \newcommand{\wsF}{\widehat{\mc F}}
 \newcommand{\wsG}{\widehat{\mc G}}
 \newcommand{\h}{\frac12}
 \newcommand{\hh}{\tfrac12}
 \newcommand{\hhh}{\frac32}
 
 \newcommand{\dex}{\frac{\rm d}{{\rm d} x}}
 \newcommand{\dt}{\text{\rm d}t}
 
 \newcommand{\dw}{\text{\rm d}w}
 \newcommand{\dx}{\text{\rm d}x}
 \newcommand{\dbx}{\text{\rm d}\bx}
 \newcommand{\dby}{\text{\rm d}\bwy}

\DeclareMathOperator{\rank}{rank}
\DeclareMathOperator{\sgn}{sgn}

\title{%
    On the number of lattice points in a ball
    }

\author{%
    Jeffrey D. Vaaler
    }

\authorinfo[%
    J. D. Vaaler]{
    Department of Mathematics, University of Texas at Austin, Austin, Texas 78712 USA}{%
    vaaler@math.utexas.edu
    }

\abstract{%
    Let $\A \subseteq \R^M$ be a discrete subgroup of rank $N$, so that $\A$ is the $\Z$-module generated by the columns of an
$M \times N$ real matrix $A$ of rank $N$.  Let $\B \subseteq \R^M$ be the $\R$-linear subspace spanned by the columns of $A$ and let $|A|$ 
denote the norm of the matrix $A$ as a linear map from $\R^N$ into $\R^M$.  We prove an explicit inequality that estimates the number of 
points in $\A$ contained in a ball of radius $R$ centered at a generic point in $\B$.  The inequality we prove is uniform over all matrices $A$ with
norm bounded by a positive constant.  A particularly simple form of the inequality occurs when $N = 3$.
    }



\VOLUME{31}
\NUMBER{2}
\YEAR{2023}
\firstpage{61}
\DOI{https://doi.org/10.46298/cm.11119}

\begin{document}



\section{Introduction}

Let $A$ be an $M\times N$ matrix with real entries and $1 \le N = \rank A \le M$. The columns of $A$ generate the free $\Z$-module
\begin{equation*}\label{intro31}
\A = \big\{A\bm : \bm\in\Z^N \big\} \subseteq \R^M
\end{equation*}
of rank $N$. The columns of $A$ also generate the $\R$-linear subspace
\begin{equation*}\label{intro38}
\B = \big\{A\bx : \bx \in \R^N \big\} \subseteq \R^M
\end{equation*}
of dimension $N$. We consider the problem of estimating the number of points in $\A$ that are contained in a~ball of positive radius $R$ centered
at a~generic point $A \bx$ in the subspace $\B$.

The general problem of estimating the number of lattice points in a~ball is treated in~\cite{kratzel1988} and~\cite{walfisz1957}. Results for lattice points
in $\R^3$ are proved in~\cite{rogerHB1999} and~\cite{tsang2000}. The papers~\cite{gotze2004},~\cite{mazo1990},~\cite{skriganov1998} and~\cite{skriganov2005} treat questions somewhat closer to the present work.

We write
\begin{equation*}\label{intro40}
|\bx| = \big(x_1^2 + x_2^2 + \dots + x_N^2 \big)^{\h}
\end{equation*}
for the Euclidean norm of a~(column) vector $\bx$ in $\R^N$, and similarly for a~vector in $\R^M$. We write
\begin{equation}\label{intro44}
V_N = \frac{\pi^{N/2}}{\Gamma(N/2 + 1)},\quad\text{and}\quad \omega_{N-1} = \frac{2\pi^{N/2}}{\Gamma(N/2)},
\end{equation}
for the volume and surface area of the unit ball, respectively, in $\R^N$. We define the normalized characteristic function
\begin{equation*}\label{intro47}
\chi_R : \R^M \rightarrow \big\{0, \hh, 1\big\}
\end{equation*}
of a~ball of radius $R$ centered at $\bo$ by
\begin{equation}\label{intro50}
\chi_R(\bwy) =
\begin{cases}
1& \text{if $|\bwy| < R$},\\
				 \h & \text{if $|\bwy| = R$},\\
					 0& \text{if $R < |\bwy|$}.
\end{cases}
\end{equation}
We suppress reference to the dimension $M$ (or to the dimension $N$) in our notation for $\chi_R$ as this should always be clear.
And we write
\begin{equation*}\label{intro54}
|A| = \sup\big\{|A\bx|: \bx\in\R^N,\ |\bx|\le 1\big\}
\end{equation*}
for the norm of the linear transformation
\begin{equation*}\label{intro61}
\bx \mapsto A\bx
\end{equation*}
from $\R^N$ into $\R^M$. Then a~more precise statement of the problem we consider is to estimate the sum
\begin{equation}\label{intro68}
\bigl(\det A^T A\bigr)^{\h} \sum_{\bm\in\Z^N} \chi_R\bigl(A(\bm + \bx)\bigr)
\end{equation}
by an expression that is independent of the point $\bx$ in $\R^N$. We seek an estimate for (\ref{intro68}) that depends
on $N$, $R$, and on an upper bound for the norm $|A|$. As
\begin{equation*}\label{intro80}
\bigl(\det A^T A\bigr)^{\h} \int_{\R^N} \chi_R\bigl(A(\bwy + \bx)\bigr)\ \dby = V_N R^N,
\end{equation*}
it is natural to estimate (\ref{intro68}) by establishing an upper bound for
\begin{equation}\label{intro87}
\biggl|\bigl(\det A^T A\bigr)^{\h} \sum_{\bm\in\Z^N} \chi_R\bigl(A(\bm + \bx)\bigr) - V_N R^N \biggr|
\end{equation}
that depends on $N$, $R$, and on an upper bound for $|A|$.

An estimate for (\ref{intro87}) follows from inequalities for extremal functions that were obtained in~\cite{HV}. The Bessel functions $J_{\nu}(x)$ and
$J_{\nu + 1}(x)$ with $2\nu + 2 = N$ naturally occur in our estimates. However, as a~convenient abuse of notation, we use both $N$ and $\nu$
depending on the situation. Our first result on the number of lattice points in a~ball is the following inequality.

\begin{theorem}\label{thmintro1} Let $A$ be an $M\times N$ matrix with real entries, and
\begin{equation*}\label{intro91}
1 \le N = \rank A \le M.
\end{equation*}
Let $-1 < \nu$ and $0 < \delta$ be real parameters such that $2\nu + 2 = N$. Then for $0 < R$ and $|A| \le \delta^{-1}$
we have
\begin{equation*}\label{intro93}
\biggl|\bigl(\det A^T A\bigr)^{1/2}\sum_{\bm\in\Z^N} \chi_R\bigl(A(\bm + \bx)\bigr) - V_N R^N \biggr| \le \omega_{N-1}u_{\nu}(R,\delta)
\end{equation*}
for all $\bx$ in $\R^N$, where $u_{\nu}(R, \delta)$ is positive and defined by
\begin{equation}\label{intro98}
u_{\nu}(R, \delta) = \delta^{-1} R^{N - 1} \bigg(1 - \tfrac{\pi}{2} (N - 1) \int_{\pi \delta R}^{\infty} x^{-1} J_{\nu}(x) J_{\nu + 1}(x)\ \dx\bigg)^{-1}.
\end{equation}
\end{theorem}

In dimension $N = 3$ we get the following simpler bound.

\begin{corollary}\label{corintro1} Let $A$ be an $M\times 3$ matrix with real entries, $\rank A = 3$, and let $0 < \delta$.
Then for $0 < R$ and $|A| \le \delta^{-1}$, we have
\begin{equation}\label{intro105}
\begin{split}
\biggl|\bigl(\det A^T A\bigr)^{1/2}\sum_{\bm\in\Z^3} \chi_R\bigl(A(\bm&+ \bx)\bigr) - \tfrac{4}{3} \pi R^3 \biggr|\\
	&\le 4 \pi \delta^{-1} R^2 \bigg(1 - \biggl(\frac{\sin \pi \delta R}{\pi \delta R}\biggr)^2 \bigg)^{-1}
\end{split}
\end{equation}
at each point $\bx$ in $\R^3$.
\end{corollary}

\begin{proof}
Bessel functions have elementary representations when the index $\nu$ is half of an odd
integer. If $N = 3$ then $\nu = \h$, and we find that
\begin{equation*}\label{intro112}
\pi x^{-1} J_{\h}(x) J_{\hhh}(x) = - \frac{d}{d x} \biggl(\frac{\sin x}{x}\biggr)^2.
\end{equation*}
The integral on the right of (\ref{intro98}) is then
\begin{equation*}
\pi \int_{\pi \delta R}^{\infty} x^{-1} J_{\h}(x) J_{\hhh}(x)\ \dx = \biggl(\frac{\sin \pi \delta R}{\pi \delta R}\biggr)^2.
\end{equation*}
Therefore (\ref{intro105}) follows from (\ref{intro98}).
\end{proof}

In section 2 we prove results that we need from the theory of entire functions of exponential type in one and several variables. We say that
an entire function $F : \C^N \rightarrow \C$ of $N$ complex variables is a~\textit{real entire} function if the restriction of $F$ to $\R^N$ takes real values.
Such functions are used throughout the paper. Section 3 contains an account of the Beurling-Selberg extremal problem for a~ball and the solution that
was found in~\cite{HV}. In particular, this includes the initial identification of $u_{\nu}(R, \delta)$ as the solution to an extremal problem.
In section 4 we establish a~special case of the Poisson summation formula (essentially~\cite[\S VII, Theorem 2.4]{stein1971}) that applies to
integrable functions $F : \R^N \rightarrow \R$ which are the restriction to $\R^N$ of a~real entire function $F : \C^N \rightarrow \C$ of exponential
type. The proof of Theorem~\ref{thmintro1} is given in section 5.

\section{Entire functions of exponential type}

We recall that an entire function $F : \C \rightarrow \C$ has \textit{exponential type} if $F$ is not identically zero, and if
\begin{equation*}\label{boas33}
\limsup_{|z| \rightarrow \infty} |z|^{-1} \log |F(z)| = \tau(F) < \infty.
\end{equation*}
If $F$ has exponential type, then the nonnegative number $\tau(F)$ is the \textit{exponential type} of $F$. We write $e(z) = e^{2\pi i z}$ for a~complex number $z$.

\begin{lemma}\label{lemboas1} Let $0 < \delta$ and let $F : \C \rightarrow \C$ be an entire function of exponential type at most $2 \pi \delta$.
Assume that
\begin{equation*}\label{boas40}
\|F\|_1 = \int_{\R} |F(x)|\ \dx < \infty,
\end{equation*}
and write $\wF : \R \rightarrow \C$ for the Fourier transform
\begin{equation}\label{boas47}
\wF(t) = \int_{\R} F(x) e(-t x)\ \dx.
\end{equation}
Then
\begin{equation}\label{boas54}
\|F\|_{\infty} = \sup \big\{|F(x)| : x \in \R\big\} \le 2 \pi \delta \|F\|_1,
\end{equation}
and
\begin{equation}\label{boas61}
\|F\|_2 = \biggl(\int_{\R} |F(x)|^2 \ \dx\biggr)^{\h}\le \bigl(2 \pi \delta\bigr)^{\h} \|F\|_1.
\end{equation}
Moreover, the Fourier transform \textrm{(\ref{boas47})} is a~continuous function supported on $[-\delta, \delta]$. And the entire function $F$ is
determined at each point $z$ in $\C$ by the Fourier inversion formula
\begin{equation}\label{boas63}
F(z) = \int_{-\delta}^{\delta} \wF(t) e(tz)\ \dt.
\end{equation}
\end{lemma}

\begin{proof}
Let $E : \C \rightarrow \C$ be the entire function defined by
\begin{equation*}\label{boas68}
E(z) = \int_0^z F(w)\ \dw.
\end{equation*}
Let $0 < \ep$ and let $C_{\ep}$ be a~positive constant that satisfies both of the inequalities
\begin{equation}\label{boas75}
|F(z)| \le C_{\ep} \exp\big((2 \pi \delta + \ep)|z|\big),\quad\text{and}\quad |z| \le C_{\ep} \exp\big(\ep |z|\big),
\end{equation}
at each point $z$ in $\C$. Using (\ref{boas75}) we find that
\begin{equation}\label{boas82}
\begin{split}
|E(z)| &= \biggl|z \int_0^1 F(t z)\ \dt\biggr| \le C_{\ep} |z| \int_0^1 \exp\big((2 \pi \delta + \ep) t |z|\big)\ \dt\\
	&\le C_{\ep}^2 \exp\big((2 \pi \delta + 2\ep) |z|\big)
\end{split}
\end{equation}
at each point $z$ in $\C$. As $0 < \ep$ is arbitrary it follows from (\ref{boas82}) that $E$ has exponential type at most $2 \pi \delta$.
At each point $x$ in $\R$ we also get the inequality
\begin{equation}\label{boas89}
|E(x)| = \biggl|\int_0^x F(w)\ \dw\biggr| \le \int_{\R} |F(w)|\ \dw = \|F\|_1,
\end{equation}
and therefore $E$ is bounded on $\R$. As $E^{\prime}(z) = F(z)$, Bernstein's inequality (see~\cite[\S 74]{achieser1956},~\cite{bernstein1923},
or~\cite{boas1954}) and (\ref{boas89}) imply that
\begin{equation*}\label{boas96}
\begin{split}
\sup \big\{|F(x)| : x \in \R\big\} &= \sup \big\{\bigl|E^{\prime}(x)\bigr| : x \in \R\big\}\\
	 &\le 2 \pi \delta \sup\big\{|E(x)| : x \in \R\big\}\\
	 &\le 2 \pi \delta \|F\|_1.
\end{split}
\end{equation*}
This proves (\ref{boas54}). Then
\begin{equation*}\label{boas103}
\begin{split}
\int_{\R} \bigl|F(x)\bigr|^2 \ \dx \le \sup \big\{|F(x)| : x \in \R\big\} \int_{\R} |F(x)|\ \dx \le 2 \pi \delta \|F\|_1^2
\end{split}
\end{equation*}
verifies the inequality (\ref{boas61}).

Because $F$ belongs to $L^1 (\R)$, it follows that the Fourier transform (\ref{boas47}) is a~function in $C_0(\R)$. Because $F$ belongs to
$L^2 (\R)$ and has exponential type at most $2 \pi \delta$, it follows from the Paley-Wiener theorem (see~\cite[\S 72]{achieser1956} or~\cite[\S 3.4]{stein1971}) that $t \mapsto \wF(t)$ is supported on the interval $[- \delta, \delta]$. Therefore the Fourier inversion formula asserts that
\begin{equation}\label{extreme5}
F(x) = \int_{-\delta}^{\delta} \wF(t) e(xt)\ \dt
\end{equation}
at each point $x$ in $\R$. Using the integral on the right of (\ref{extreme5}) and Morera's theorem it is easy to show that
\begin{equation}\label{extreme10}
z \mapsto \int_{-\delta}^{\delta} \wF(t) e(z t)\ \dt
\end{equation}
defines an entire function of $z = x + iy$. Plainly the entire function defined by (\ref{extreme10}) is equal to the entire function
$z \mapsto F(z)$ for $z$ in $\R$. Therefore we get
\begin{equation*}\label{extreme15}
F(z) = \int_{-\delta}^{\delta} \wF(t) e(z t)\ \dt
\end{equation*}
at each point $z$ in $\C$ by analytic continuation. This verifies (\ref{boas63}).
\end{proof}

We write $\bz = (z_n)$ for a~(column) vector in $\C^N$ and
\begin{equation*}\label{intro181}
|\bz| = \bigl(|z_1|^2 + |z_2|^2 + \cdots + |z_N|^2 \bigr)^{\h}
\end{equation*}
for the usual Hermitian norm of $\bz$. We define a~second norm
\begin{equation*}\label{extreme53}
\|\ \| : \C^N \rightarrow [0, \infty)
\end{equation*}
on vectors $\bz$ in $\C^N$ by setting
\begin{equation*}\label{extreme57}
\|\bz\| = \sup\Big\{\bigl|z_1 t_1 + z_2 t_2 + \cdots + z_N t_N\bigr| : \text{$\bt \in \R^N$ and $|\bt| \le 1$}\Big\}.
\end{equation*}
If $N = 1$ then $|z| = \|z\|$ at each point $z$ in $\C$. But if $2 \le N$ we find that
\begin{equation}\label{extreme60}
\|\bz\|^2 \le |\bz|^2 \le 2 \|\bz\|^2
\end{equation}
at each point $\bz$ in $\C^N$, where both inequalities in (\ref{extreme60}) are sharp. If
\begin{equation*}\label{extreme62}
F : \C^N \rightarrow \C
\end{equation*}
is an entire function of $N$ complex variables and not identically zero, we say that $F$ has \textit{exponential type} if
\begin{equation*}\label{extreme64}
\limsup_{\|\bz\| \rightarrow \infty} \|\bz\|^{-1} \log |F(\bz)| = \tau(F) < \infty.
\end{equation*}
The nonnegative number $\tau(F)$ is the \textit{exponential type} of $F$. If $N = 1$ this is the usual definition of exponential type discussed above.
If $2 \le N$ our definition is a~special case of a~more general notion of exponential type introduced by Stein in~\cite{stein1957}, (see also~\cite[pp. 111-112]{stein1971}).

Next we define a~map that sends an even entire function $F : \C \rightarrow \C$ having exponential type into a~radial entire function
\begin{equation*}\label{boas110}
\psi_N(F) : \C^N \rightarrow \C
\end{equation*}
having exponential type in the sense of Stein~\cite{stein1957}. This map was defined and used in~\cite[section 6]{HV}.

Let $F : \C \rightarrow \C$ be an even entire function. Then the power series for $F$ at $0$ can be written as
\begin{equation}\label{extreme30}
F(z) = \sum_{m = 0}^{\infty} c_{2m}(F) z^{2m},
\end{equation}
where
\begin{equation*}\label{extreme35}
c_{2m}(F) = \frac{F^{(2m)}(0)}{(2m)!}\quad\text{for $m = 0, 1, 2, \dots $.}
\end{equation*}
As is well known, the partial sums for the series in (\ref{extreme30}) converge uniformly on compact subsets of $\C$. For each positive integer $N$ we define,
as in~\cite[Lemma 18]{HV}, the entire function
\begin{equation*}\label{extreme40}
\psi_N(F) : \C^N \rightarrow \C
\end{equation*}
of $N$ complex variables $z_1, z_2, \dots, z_N$, by
\begin{equation}\label{extreme45}
\psi_N(F)(\bz) = \sum_{m = 0}^{\infty} \frac{c_{2m}(F) \bigl(z_1^2 + z_2^2 + \cdots + z_N^2 \bigr)^m}{(2m)!}.
\end{equation}
If $\bx$ belongs to $\R^N$ we find that
\begin{equation}\label{extreme52}
\begin{split}
\psi_N(F)(\bx) &= \sum_{m = 0}^{\infty} \frac{c_{2m}(F) \bigl(x_1^2 + x_2^2 + \cdots + x_N^2 \bigr)^m}{(2m)!}\\
                       &= \sum_{m = 0}^{\infty} \frac{c_{2m}(F) |\bx|^{2m}}{(2m)!}\\
                       &= F\bigl(|\bx|\bigr).
\end{split}
\end{equation}
It follows that $\psi_N(F)$ restricted to $\R^N$ is a~radial function.

\begin{lemma}\label{lemboas2} Let $F : \C \rightarrow \C$ be an even entire function, and for each positive integer $N$ let
\begin{equation*}\label{boas131}
\psi_N(F) : \C^N \rightarrow \C
\end{equation*}
be the entire function of $N$ complex variables defined by \textrm{(\ref{extreme45})}. Define
\begin{equation*}\label{boas145}
\|F\|_{\infty} = \sup\big\{|F(x)| : x \in \R\big\},
\end{equation*}
and
\begin{equation*}\label{boas152}
\big\|\psi_N(F)\big\|_{\infty} = \sup\big\{\bigl|\psi_N(F)(\bx)\bigr| : \bx \in \R^N \big\}.
\end{equation*}
Then $F$ has exponential type if and only if $\psi_N(F)$ has exponential type. Moreover, if either $F$ or $\psi_N(F)$ has exponential type, then
\begin{equation}\label{boas159}
\tau(F) = \tau\bigl(\psi_N(F)\bigr).
\end{equation}
Also, $F$ restricted to $\R$ is bounded if and only if $\psi_N(F)$ restricted to $\R^N$ is bounded. Additionally, if either $F$ is bounded on $\R$ or
$\psi_N(F)$ is bounded on $\R^N$, then
\begin{equation}\label{boas166}
\|F\|_{\infty} = \big\|\psi_N(F)\big\|_{\infty}.
\end{equation}
\end{lemma}

\begin{proof}
That $F$ has exponential type if and only if $\psi_N(F)$ has exponential type, together with the identity (\ref{boas159}), were both
established in~\cite[Lemma 18]{HV}.

Because $F$ is an even function, it follows from the identity (\ref{extreme52}) that
\begin{equation*}\label{boas173}
\big\{F(x) : x \in \R\big\} = \big\{F\bigl(|\bx|\bigr) : \bx \in \R^N \big\} = \big\{\psi_N(F)(\bx) : \bx \in \R^N \big\},
\end{equation*}
and therefore
\begin{equation}\label{boas180}
\big\{\big|F(x)\bigr| : x \in \R\big\} = \big\{\bigl|\psi_N(F)(\bx)\bigr| : \bx \in \R^N \big\}.
\end{equation}
The conclusion that $F$ is bounded on $\R$ if and only if $\psi_N(F)$ is bounded on $\R^N$, and the proposed identity (\ref{boas166}), follow
immediately from (\ref{boas180}).
\end{proof}

If we assume that the even, entire function
\begin{equation*}\label{extreme67}
F : \C \rightarrow \C
\end{equation*}
has exponential type at most $2 \pi \delta$, then it follows from Lemma~\ref{lemboas2} that the entire function $\psi_N(F)(\bz)$, defined using
(\ref{extreme30}) and (\ref{extreme45}), has exponential type at most $2 \pi \delta$. Moreover,~\cite[equation (6.3)]{HV} asserts that
\begin{equation}\label{extreme70}
\hh \omega_{N-1} \int_{\R} |F(x)||x|^{2\nu + 1}\ \dx = \int_{\R^N} |\psi_N(F)(\bx)||\bx|^{2\nu + 2 - N}\ \dbx,
\end{equation}
where $-1 < \nu$. And if either of the integrals in (\ref{extreme70})
is finite then both integrals are finite, and we get (this is~\cite[equation (6.4)]{HV})
\begin{equation*}\label{extreme75}
\hh \omega_{N-1} \int_{\R} F(x) |x|^{2\nu + 1}\ \dx = \int_{\R^N} \psi_N(F)(\bx) |\bx|^{2\nu + 2 - N}\ \dbx.
\end{equation*}
The special case $2 \nu + 2 = N$ leads to the following result.

\begin{lemma}\label{lemboas3} Let $F : \C \rightarrow \C$ be an even entire function of exponential type at most $2 \pi \delta$. Assume that $N$ is a
positive integer such that
\begin{equation}\label{boas198}
\int_{\R} |F(x)||x|^{N - 1}\ \dx < \infty,
\end{equation}
and let
\begin{equation*}\label{boas191}
\psi_N(F) : \C^N \rightarrow \C
\end{equation*}
be the entire function of $N$ complex variables defined by \textrm{(\ref{extreme45})}. Then the restriction of $\psi_N(F)$ to $\R^N$ belongs to
\begin{equation*}\label{boas205}
L^1 \bigl(\R^N \bigr) \cap L^2 \bigl(\R^N \bigr).
\end{equation*}
Moreover, the Fourier transform
\begin{equation*}\label{boas212}
\widehat{\psi_N(F)} : \R^N \rightarrow \C
\end{equation*}
defined by
\begin{equation}\label{boas219}
\widehat{\psi_N(F)}(\bt) = \int_{\R^N} \psi_N(F)(\bx) e\bigl(-\bt^T \bx\bigr)\ \dbx,
\end{equation}
is a~continuous, radial function supported in the closed ball
\begin{equation}\label{boas226}
\big\{\bt \in \R^N : |\bt| \le \delta\big\}.
\end{equation}
\end{lemma}

\begin{proof}
By taking $2 \nu + 2 = N$ in (\ref{extreme70}), we find that
\begin{equation}\label{boas240}
\hh \omega_{N-1} \int_{\R} |F(x)| |x|^{N - 1}\ \dx = \int_{\R^N} |\psi_N(F)(\bx)|\ \dbx < \infty.
\end{equation}
Therefore $\psi_N(F)$ belongs to $L^1 \bigl(\R^N \bigr)$, and the Fourier transform (\ref{boas219}) is a~continuous, radial function
(see~\cite[\S IV, Corollary 1.2]{stein1971}).

Clearly (\ref{boas198}) implies that $F$ belongs to $L^1 (\R)$. Then using (\ref{boas54}) and (\ref{boas166}) we get
\begin{equation}\label{boas247}
\big\|\psi_N(F)\big\|_{\infty} = \|F\|_{\infty} \le 2 \pi \delta \|F\|_1 < \infty.
\end{equation}
It follows from (\ref{boas240}) and (\ref{boas247}) that
\begin{equation*}\label{boas254}
\int_{\R^N} \bigl|\psi_N(F)(\bx)\bigr|^2 \ \dbx \le \big\|\psi_N(F)\big\|_{\infty} \int_{\R^N} \bigl|\psi_N(F)(\bx)\bigr|\ \dbx < \infty.
\end{equation*}
This verifies that $\psi_N(F)$ belongs to both $L^1 \bigl(\R^N \bigr)$ and $L^2 \bigl(\R^N \bigr)$. Hence the Fourier transform (\ref{boas219}) is a~continuous,
radial function. And by the generalization of the Paley-Wiener theorem proved by Stein (see~\cite[\S III, Theorem 4]{stein1957}
or~\cite[\S III Theorem 4.9]{stein1971}), we conclude that (\ref{boas219}) is supported on the closed ball (\ref{boas226}).
\end{proof}

\section{Beurling-Selberg extremal problems}

In this section we consider extremal problems first investigated by A. Beurling and later by A. Selberg and we describe the solution to a~more general extremal
problem that was obtained in~\cite{HV}. The method used in~\cite{HV} was based on earlier work of L. de Branges~\cite{deBranges1962},~\cite{deBranges1968}.
Further information about these problems can be found in~\cite{Logan1985},~\cite{M},~\cite{S1},~\cite{S2}, and~\cite{vaaler1985}.

Let $\xi$, $\delta$, and $\nu$ be real numbers that satisfy $0 < \delta$ and $-1 < \nu$. Then let $z \mapsto S(z)$ and $z \mapsto T(z)$ be real entire
functions such that
\begin{equation}\label{deB400}
\text{$S(x) \le \sgn(x - \xi) \le T(x)$ for all real $x$,}
\end{equation}
and
\begin{equation}\label{deB404}
\text{$z \mapsto S(z)$ and $z \mapsto T(z)$ have exponential type at most $2\pi\delta$.}
\end{equation}
We define the real number $u_{\nu}(\xi, \delta)$ to be the infimum of the collection of positive numbers
\begin{equation}\label{deB407}
\hh \int_{\R} \big(T(x) - S(x)\big) |x|^{2\nu + 1}\ \dx
\end{equation}
taken over the set of all pairs of real entire functions $S(z)$ and $T(z)$ that satisfy (\ref{deB400}) and (\ref{deB404}), and for which the value of the
integral in (\ref{deB407}) is finite. The problem considered by Beurling was the special case $\xi = 0$ and $\nu = -\h$. Related extremal problems
are considered in~\cite{CV2},~\cite{CV3},~\cite{CG2015},~\cite{CLV},~\cite[\S 1.1 and \S 1.2]{CL2014},~\cite[\S 1.2 and \S 1.4]{CL2017},~\cite{CL2018},~\cite{Lit0},~\cite{Lit1},~\cite{Lit2}, and~\cite{Lit3}.

In~\cite[Theorem 1]{HV} the authors proved that for $0 < \delta$ and $-1 < \nu$, the infimum $u_{\nu}(\xi, \delta)$ is positive.
They also proved that there exists a~\textit{unique} pair of real entire functions
\begin{equation}\label{deB410}
z \mapsto s_{\nu}(z; \xi, \delta)\quad\text{and}\quad z \mapsto t_{\nu}(z; \xi, \delta)
\end{equation}
that satisfy the inequality
\begin{equation}\label{deB411}
s_{\nu}(x; \xi, \delta) \le \sgn(x - \xi) \le t_{\nu}(x; \xi, \delta) \quad\text{for all real $x$},
\end{equation}
the functions
\begin{equation}\label{deB412}
\text{$z \mapsto s_{\nu}(z; \xi, \delta)$ and $z \mapsto t_{\nu}(z; \xi, \delta)$ have exponential type at most $2 \pi \delta$},
\end{equation}
and satisfy the identity
\begin{equation}\label{deB414}
u_{\nu}(\xi, \delta) = \hh\int_{\R} \big(t_{\nu}(x; \xi, \delta) - s_{\nu}(x; \xi, \delta)\big) |x|^{2\nu + 1}\ \dx.
\end{equation}
The functions (\ref{deB410}) are defined in~\cite[equation (5.7) and (5.8)]{HV}.

In~\cite[Theorem 1]{HV} the authors proved that $u_{\nu}(\xi, \delta)$ satisfies the following identities:
\begin{itemize}
\item[(i)] $u_{\nu}(\xi, \delta) = u_{\nu}(-\xi, \delta)$,
\item[(ii)] if $0 < \kappa$ then $u_{\nu}(\xi, \delta) = \kappa^{2\nu + 2} u_{\nu}\bigl(\kappa^{-1} \xi, \kappa \delta\bigr)$,
\item[(iii)] $u_{\nu}(0, \delta) = \Gamma(\nu + 1) \Gamma(\nu + 2)\bigl(2/\pi \delta\bigr)^{2\nu + 2}$,
\item[(iv)] if $0 < \xi$ then (this is~\cite[equation (1.4)]{HV})
\begin{equation}\label{deB419}
u_{\nu}\bigl(\xi, \pi^{-1}\bigr) = \frac{2\xi^{2\nu + 1}}{\xi J_{\nu}(\xi)^2 + \xi J_{\nu + 1}(\xi)^2 - (2\nu + 1) J_{\nu}(\xi) J_{\nu + 1}(\xi)},
\end{equation}
where $z \mapsto J_{\nu}(z)$ and $z \mapsto J_{\nu + 1}(z)$ are Bessel functions.
\end{itemize}

Using elementary properties of Bessel functions (see~\cite{lommel1868} and~\cite[\S 3.2]{watson1944}) we find that
\begin{equation}\label{deB423}
\lim_{\xi \rightarrow \infty} \xi J_{\nu}(\xi)^2 + \xi J_{\nu + 1}(\xi)^2 - (2\nu + 1) J_{\nu}(\xi) J_{\nu + 1}(\xi) = 2 \pi^{-1},
\end{equation}
and
\begin{equation}\label{deB430}
\begin{split}
\dex \bigg(x J_{\nu}(x)^2 + x J_{\nu + 1}(x)^2 -&(2\nu + 1) J_{\nu}(x) J_{\nu + 1}(x)\bigg)\\
		&= (2\nu + 1) x^{-1} J_{\nu}(x) J_{\nu + 1}(x).
\end{split}
\end{equation}
If $0 < \xi$ then (\ref{deB423}) and (\ref{deB430}) lead to the identity
\begin{equation}\label{deB435}
    \begin{split}
\xi J_{\nu}(\xi)^2 &+ \xi J_{\nu + 1}(\xi)^2 -(2\nu + 1) J_{\nu}(\xi) J_{\nu + 1}(\xi)\\
		&= 2 \pi^{-1} - (2\nu + 1) \int_{\xi}^{\infty} x^{-1} J_{\nu}(x) J_{\nu + 1}(x)\ \dx.
\end{split}
\end{equation}
Combining (\ref{deB419}) and (\ref{deB435}) we find that
\begin{equation*}\label{deB442}
u_{\nu}\bigl(\xi, \pi^{-1}\bigr) = \xi^{2\nu + 1} \bigg(\pi^{-1} - \bigl(\nu + \hh\bigr) \int_{\xi}^{\infty} x^{-1} J_{\nu}(x) J_{\nu + 1}(x)\ \dx\bigg)^{-1}.
\end{equation*}
Then by applying the identity (ii) with $\kappa = (\pi \delta)^{-1}$, we get the general formula
\begin{equation}\label{deB447}
u_{\nu}\bigl(\xi, \delta\bigr) = \delta^{-1} \xi^{2\nu + 1} \bigg(1 - \tfrac{\pi}{2} (2\nu + 1) \int_{\pi \delta \xi}^{\infty} x^{-1} J_{\nu}(x) J_{\nu + 1}(x)\ \dx\bigg)^{-1},
\end{equation}
which holds for $0 < \xi$, $0 < \delta$ and $-1 < \nu$. In the particular case where $2\nu + 2 = N$ the map (\ref{deB447}) is the same as (\ref{intro98}) in the statement of Theorem~\ref{thmintro1}.

In the special case $\nu = -\h$ and $\xi = 0$ considered by Beurling, the two extremal functions (\ref{deB410}) can be represented using interpolation
formulas that were given in~\cite[Theorem 9 and Theorem 10]{vaaler1985}.

Next we describe an extremal problem for a~ball in $\R^N$ centered at $\bo$ that was considered in~\cite{HV}. Related results can be found in~\cite{felipe2018}.

Let $R$, $\delta$, and $\nu$, be real numbers such that $0 < R$, $0 < \delta$, and $-1 < \nu$. Then let
\begin{equation}\label{intro191}
F : \C^N \rightarrow \C,\quad\text{and}\quad G : \C^N \rightarrow \C,
\end{equation}
be real entire functions such that
\begin{equation}\label{intro192}
F(\bx) \le \chi_R(\bx) \le G(\bx)\quad\text{at each point $\bx$ in $\R^N$},
\end{equation}
and
\begin{equation}\label{intro193}
\text{$F(\bz)$ and $G(\bz)$ have exponential type at most $2 \pi \delta$}.
\end{equation}
Define $H_{\nu}(N, R, \delta)$ to be the infimum of positive numbers
\begin{equation}\label{intro197}
\int_{\R^N} \big(G(\bx) - F(\bx)\big) |\bx|^{2\nu + 2 - N}\ \dbx
\end{equation}
taken over the set of all ordered pairs $(F, G)$ of real entire functions (\ref{intro191}) that satisfy (\ref{intro192}), (\ref{intro193}), and for which the
integral (\ref{intro197}) is finite.

In the special case $\nu = -\h$ and $N = 1$ the problem of evaluating (or estimating) the infimum $H_{-\h}(1, R, \delta)$ was considered by Selberg~\cite{S1},~\cite{S2}. For $N = 1$ the simple identity
\begin{equation}\label{intro200}
\chi_R(x) = \hh\bigl(\sgn(x + R) - \sgn(x - R)\bigr),
\end{equation}
which holds for all real $x$, indicates the connection between Beurling's extremal problem and the extremal problem considered by Selberg.
Selberg observed (at least in the case $\nu = -\h$) that for all real $x$ the inequalities
\begin{equation}\label{intro202}
\hh\big(s_{\nu}(x; -R, \delta) - t_{\nu}(x; R, \delta)\big) \le \chi_R(x) \le \hh\big(t_{\nu}(x; -R, \delta) - s_{\nu}(x; R, \delta)\big),
\end{equation}
follow from the two inequalities in (\ref{deB411}) and the identity (\ref{intro200}). Therefore we define the two real entire functions
\begin{equation}\label{intro218}
f_{\nu}(z; R, \delta) = \hh\big(s_{\nu}(z; -R, \delta) - t_{\nu}(z; R, \delta)\big)
\end{equation}
and
\begin{equation}\label{intro222}
g_{\nu}(z; R, \delta) = \hh\big(t_{\nu}(z; -R, \delta) - s_{\nu}(z; R, \delta)\big).
\end{equation}
It follows from (\ref{deB412}) that both $z \mapsto f_{\nu}(z; R, \delta)$ and $z \mapsto g_{\nu}(z; R, \delta)$ have exponential type at most $2 \pi \delta$.
Clearly the inequality (\ref{intro202}) is also
\begin{equation*}\label{intro227}
f_{\nu}(x; R, \delta) \le \chi_R(x) \le g_{\nu}(x; R, \delta)
\end{equation*}
for all real $x$ and all positive values of $R$.

From (\ref{deB411}) we find that \textit{both} of the inequalities
\begin{equation*}\label{intro205}
s_{\nu}(x; R, \delta) \le \sgn(x - R) \le t_{\nu}(x; R, \delta) \quad\text{for all real $x$},
\end{equation*}
and
\begin{equation*}\label{intro207}
-t_{\nu}(-x; R, \delta) \le \sgn(x + R) \le -s_{\nu}(-x; R, \delta) \quad\text{for all real $x$},
\end{equation*}
must hold. As the extremal functions $x \mapsto s_{\nu}(x; R, \delta)$ and $x \mapsto t_{\nu}(x; R, \delta)$ that satisfy (\ref{deB411}), (\ref{deB412}), and
(\ref{deB414}) are \textit{unique}, we also get
\begin{equation}\label{intro214}
-s_{\nu}(-x; R, \delta) = t_{\nu}(x; -R, \delta)
\end{equation}
for all real $x$. Then it follows from (\ref{intro214}) that both (\ref{intro218}) and (\ref{intro222}) are \textit{even} entire functions.

Next we describe the solution to the problem of evaluating (or estimating) the infimum $H_{\nu}(N, R, \delta)$ for a~ball in $\R^N$ which was obtained
in~\cite{HV}. We write
\begin{equation}\label{intro232}
\bz \mapsto \F_{\nu}(\bz; R, \delta),\quad\text{and}\quad \bz \mapsto \G_{\nu}(\bz; R, \delta),
\end{equation}
for the two real entire functions defined in~\cite[equation (1.22) and (1.23)]{HV}. From our discussion of the maps (\ref{extreme45}) and (\ref{extreme52}),
it follows that the functions (\ref{intro232}) are also given by
\begin{equation}\label{intro238}
\F_{\nu}(\bz; R, \delta) = \psi_N\bigl(f_{\nu}\bigr)(\bz; R, \delta),\quad\text{and}\quad \G_{\nu}(\bz; R, \delta) = \psi_N\bigl(g_{\nu}\bigr)(\bz; R, \delta),
\end{equation}
where $f_{\nu}$ and $g_{\nu}$ are the even functions defined in (\ref{intro218}) and (\ref{intro222}). It follows from~\cite[Theorem 3]{HV} that the
functions (\ref{intro232}) are both real entire functions of exponential type at most $2 \pi \delta$. And it follows from the representations (\ref{intro238})
that the restriction of these functions to $\R^N$ are both radial functions that satisfy the inequalities
\begin{equation}\label{intro249}
\F_{\nu}(\bx; R,\delta) \le \chi_R(\bx) \le \G_{\nu}(\bx; R,\delta)
\end{equation}
for all $\bx$ in $\R^N$. Moreover, these functions satisfy the integral identity (this is~\cite[equation (1.27)]{HV})
\begin{equation}\label{intro253}
\omega_{N-1} u_{\nu}(R, \delta) = \int_{\R^N} \big(\G_{\nu}(\bx; R, \delta) - \F_{\nu}(\bx; R, \delta)\big) |\bx|^{2\nu + 2 - N}\ \dbx,
\end{equation}
where $\omega_{N-1}$ (the surface area of a~unit ball in $\R^N$) was defined in (\ref{intro44}). Now (\ref{intro253}) (this is~\cite[equation (1.25)]{HV})
implies that
\begin{equation}\label{intro257}
H_{\nu}(N, R, \delta) \le \omega_{N-1} u_{\nu}(R, \delta).
\end{equation}
It was also shown in~\cite[Theorem 3]{HV} that there is equality in the inequality (\ref{intro257}) if and only if
\begin{equation*}\label{intro259}
J_{\nu}(\pi \delta R) J_{\nu + 1}(\pi \delta R) = 0.
\end{equation*}

We now restrict our attention to the case $2 \nu + 2 = N$.

\begin{lemma}\label{lemintro1} Let $R$, $\delta$ and $\nu$ be real numbers such that $0 < R$, $0 < \delta$, and $-1 < \nu$, and let $N$ be
a positive integer such that $2 \nu + 2 = N$. Let
\begin{equation*}\label{intro263}
\F_{\nu}(\bz; R, \delta) = \psi_N\bigl(f_{\nu}\bigr)(\bz; R, \delta),\quad\text{and}\quad \G_{\nu}(\bz; R, \delta) = \psi_N\bigl(g_{\nu}\bigr)(\bz; R, \delta),
\end{equation*}
be the real entire functions defined by \textrm{(\ref{intro238})}, or defined in~\cite[equation (1.22) and (1.23)]{HV}. Then both the restriction of
$\F_{\nu}(\bz; R, \delta)$ to $\R^N$, and the restriction of $\G_{\nu}(\bz; R, \delta)$ to $\R^N$, belong to
\begin{equation}\label{intro265}
L^1 \bigl(\R^N \bigr) \cap L^2 \bigl(\R^N \bigr).
\end{equation}
Moreover, both of the Fourier transforms
\begin{equation*}\label{intro267}
\wsF_{\nu} : \R^N \rightarrow \C,\quad\text{and}\quad \wsG_{\nu} : \R^N \rightarrow \C,
\end{equation*}
defined by
\begin{equation}\label{intro269}
\wsF_{\nu}(\bt; R,\delta) = \int_{\R^N} \F_{\nu}(\bz; R,\delta) e\bigl(-\bt^T \bx\bigr)\ \dbx,
\end{equation}
and
\begin{equation}\label{intro271}
\wsG_{\nu}(\bt; R,\delta) = \int_{\R^N} \G_{\nu}(\bz; R,\delta) e\bigl(-\bt^T \bx\bigr)\ \dbx,
\end{equation}
are continuous functions supported on the closed ball
\begin{equation*}\label{intro273}
\big\{\bt \in \R^N : |\bt| \le \delta\big\}.
\end{equation*}
\end{lemma}

\begin{proof}
At each point $\bx$ in $\R^N$ the restrictions of $\F_{\nu}$ and $\G_{\nu}$ to $\R^N$ satisfy the inequality (\ref{intro249}).
It follows from (\ref{intro253}), and our assumption that $2 \nu + 2 = N$, that
\begin{equation}\label{intro278}
\int_{\R^N} \big(\G_{\nu}(\bx; R, \delta) - \F_{\nu}(\bx; R, \delta)\big)\ \dbx < \infty.
\end{equation}
Let
$\F_{\nu}^+ : \R^N \rightarrow [0, \infty)$ and $ \F_{\nu}^- : \R^N \rightarrow [0, \infty)$
be defined by
\begin{equation*}
\F_{\nu}^+ (\bx; R, \delta) = \max\big\{0, \F_{\nu}(\bx; R, \delta)\big\}, \text{ and }
\F_{\nu}^- (\bx; R, \delta) = \max\big\{0, - \F_{\nu}(\bx; R, \delta)\big\}.
\end{equation*}
It follows that
\begin{equation*}
\F_{\nu}(\bx; R, \delta) = \F_{\nu}^+ (\bx; R, \delta) - \F_{\nu}^- (\bx; R, \delta)
\end{equation*}
and
\begin{equation*}
\bigl|\F_{\nu}(\bx; R, \delta)\bigr| = \F_{\nu}^+ (\bx; R, \delta) + \F_{\nu}^- (\bx; R, \delta).
\end{equation*}
From (\ref{intro249}) we also get the inequality
\begin{equation}\label{intro293}
\F_{\nu}^+ (\bx; R, \delta) \le \chi_R(\bx)
\end{equation}
at each point $\bx$ in $\R^N$. Then it follows using (\ref{intro293}) that
\begin{equation}\label{intro295}
\begin{split}
\bigl|\G_{\nu}(\bx; R, \delta)\bigr|&+ \bigl|\F_{\nu}(\bx; R, \delta)\bigr|\\
	&= \G_{\nu}(\bx; R, \delta) + \F_{\nu}^+ (\bx; R, \delta) + \F_{\nu}^- (\bx; R, \delta)\\
	&= \G_{\nu}(\bx; R, \delta) - \F_{\nu}(\bx; R, \delta) + 2 \F_{\nu}^+ (\bx; R, \delta)\\
	&\le \G_{\nu}(\bx; R, \delta) - \F_{\nu}(\bx; R, \delta) + 2 \chi_R(\bx)
\end{split}
\end{equation}
at each point $\bx$ in $\R^N$. From (\ref{intro249}) and (\ref{intro278}) we find that the nonnegative valued function
\begin{equation*}\label{intro296}
\bx \mapsto \G_{\nu}(\bx; R, \delta) - \F_{\nu}(\bx; R, \delta)
\end{equation*}
belongs to $L^1 \bigl(\R^N \bigr)$, and it is obvious that
\begin{equation*}\label{intro299}
\bx \mapsto \chi_R(\bx)
\end{equation*}
also belongs to $L^1 \bigl(\R^N \bigr)$. Then (\ref{intro295}) implies that both the restriction of
$\F_{\nu}(\bz; R, \delta)$ to $\R^N$ and the restriction of $\G_{\nu}(\bz; R, \delta)$ to $\R^N$ belong to $L^1 \bigl(\R^N \bigr)$.

Let
\begin{equation}\label{intro301}
f_{\nu} : \C \rightarrow \C,\quad\text{and}\quad g_{\nu} : \C \rightarrow \C
\end{equation}
be the real entire functions of exponential type at most $2 \pi \delta$ defined by (\ref{intro218}) and (\ref{intro222}).
The functions $\F_{\nu}$ and $G_{\nu}$ are the image of $f_{\nu}$ and $g_{\nu}$, respectively, with respect to the map $\psi_N$. That is, we have
\begin{equation*}\label{intro303}
\F_{\nu}(\bz; R, \delta) = \psi_N\bigl(f_{\nu}\bigr)(\bz; R, \delta),\quad\text{and}\quad \G_{\nu}(\bz; R, \delta) = \psi_N\bigl(g_{\nu}\bigr)(\bz; R, \delta).
\end{equation*}
It follows from (\ref{extreme70}), our assumption that $N = 2 \nu + 2$, and what we have already proved, that both
\begin{equation*}\label{intro305}
\hh \omega_{N-1} \int_{\R} |f_{\nu}(x; R, \delta)||x|^{N - 1}\ \dx = \int_{\R^N} |\F_{\nu}(\bx; R, \delta)|\ \dbx < \infty,
\end{equation*}
and
\begin{equation*}\label{extreme307}
\hh \omega_{N-1} \int_{\R} |g_{\nu}(x; R, \delta)||x|^{N - 1}\ \dx = \int_{\R^N} |\G_{\nu}(\bx; R, \delta)|\ \dbx < \infty.
\end{equation*}
Therefore the two functions (\ref{intro301}) satisfy the hypotheses of Lemma~\ref{lemboas3}. Then it follows from the conclusion of Lemma~\ref{lemboas3}
that the functions $\F_{\nu} = \psi_N\bigl(f_{\nu}\bigr)$ and $\G_{\nu}\bigl(g_{\nu}\bigr)$ belong to the set (\ref{intro265}). The properties attributed to
the Fourier transforms (\ref{intro269}) and (\ref{intro271}) also follow from Lemma~\ref{lemboas3}.
\end{proof}

\section{The Poisson formula}

We require the following elementary lemma.

\begin{lemma}\label{lemappl1}
Let $A$ be an $M\times N$ matrix with entries in $\R$ and
\begin{equation*}\label{intro94}
1 \le N = \rank A \le M.
\end{equation*}
Then there exists a~real $N\times N$, positive definite, symmetric matrix $S$ such
\begin{equation}\label{intro101}
|A\bx| = |S\bx|
\end{equation}
for all vectors $\bx$ in $\R^N$. Moreover, we have
\begin{equation}\label{intro108}
|A|^{-1} \le \min\big\{\bigl|S^{-1}\bn\bigr|: \bn\in\Z^N,\ \bn\not= \bo\big\}.
\end{equation}
\end{lemma}

\begin{proof}
The matrix $A^T A$ is $N\times N$, positive definite and symmetric. Hence there exists (see~\cite[Theorem 8.6.10]{artin2011}) an $N\times N$ orthogonal matrix $\Phi$ and an $N\times N$ diagonal matrix $D = [d_n]$ with positive
diagonal entries $d_1, d_2, \dots, d_N$, such that
\begin{equation*}\label{intro115}
A^T A = \Phi^T D\Phi.
\end{equation*}
We set
\begin{equation*}\label{intro122}
S = \Phi^T D^{\h}\Phi,\quad\text{where}\quad D^{\h} = \bigl[d_n^{\h}\bigr].
\end{equation*}
It follows that $S$ is an $N\times N$, positive definite, symmetric matrix, and
\begin{equation*}\label{intro129}
\begin{split}
|S\bx|^2 &= \bx^T S^T S\bx = \bx^T \Phi^T D^{\h}\Phi\Phi^T D^{\h}\Phi\bx\\
			 &= \bx^T A^T A\bx = |A\bx|^2.
\end{split}
\end{equation*}
for all vectors $\bx$ in $\R^N$. This establishes (\ref{intro101}).

As is well known, we have
\begin{equation}\label{intro136}
\begin{split}
|A|^2 &= \sup\big\{\bx^T \Phi^T D\Phi\bx: \bx\in\R^N,\ |\bx|\le 1\big\}\\
		 &= \sup\big\{\bwy^T D\bwy: \bwy\in\R^N,\ |\bwy|\le 1\big\}\\
		 &= \sup\Big\{\sum_{n=1}^N d_ny_n^2 : \bwy\in\R^N,\ |\bwy|\le 1\Big\}\\
		 &= \max\big\{d_n: 1\le n\le N\big\}\\
		 &= \Bigl(\min\big\{d_n^{-1}: 1\le n\le N\big\}\Bigr)^{-1}.
\end{split}
\end{equation}
Let $\bn\not= \bo$ belong to $\Z^N$ and let $\be_1, \be_2, \dots, \be_N$ denote the standard basis vectors in $\Z^N \subseteq \R^N$. Then
\begin{equation}\label{intro143}
|\Phi\bn|^2 = \sum_{n = 1}^N \bigl(\be_n^T \Phi \bn\bigr)^2 = |\bn|^2 \ge 1,
\end{equation}
and it follows from (\ref{intro136}) and (\ref{intro143}) that
\begin{equation*}\label{intro150}
\begin{split}
\bigl|S^{-1}\bn\bigr|^2 &= \bn^T \Phi^T D^{-1}\Phi\bn\\
				 &= \sum_{n=1}^N d_n^{-1}\bigl(\be_n^T \Phi\bn\bigr)^2 \\
				 &\ge \Bigl(\min\big\{d_n^{-1}: 1\le n\le N\big\}\Bigr) |\Phi \bn|^2 \\
				 &\ge |A|^{-2}.
\end{split}
\end{equation*}
This shows that the inequality (\ref{intro108}) holds.
\end{proof}

Let $A$ and $S$ be real matrices as in Lemma~\ref{lemappl1}, which satisfy (\ref{intro101}) and (\ref{intro108}). Then we have
\begin{equation}\label{intro154}
\bigl(\det A^T A\bigr)^{\h} \sum_{\bm\in\Z^N} \chi_R\bigl(A(\bm + \bx)\bigr) = (\det S) \sum_{\bn \in \Z^N} \chi_R\bigl(S(\bn + \bx)\bigr),
\end{equation}
where $\bwy \mapsto \chi_R(\bwy)$ is the normalized characteristic function defined in (\ref{intro50}). It is obvious that the function defined by
\begin{equation*}\label{intro159}
\bx \mapsto (\det S) \sum_{\bn\in\Z^N} \chi_R\bigl(S(\bn + \bx)\bigr) - V_N R^N
\end{equation*}
is constant on cosets of the quotient group $\R^N /\Z^N$.

Suppose that $F:\C^N \rightarrow \C$ is a~real entire function of exponential type at most $2 \pi \delta$. We also assume that the restriction
of $F$ to $\R^N$ is a~radial function that belongs to
\begin{equation*}\label{intro162}
L^1 \bigl(\R^N \bigr) \cap L^2 \bigl(\R^N \bigr).
\end{equation*}
It follows that the Fourier transform
\begin{equation*}\label{intro164}
\wF:\R^N \rightarrow \C
\end{equation*}
is a~continuous function, and it follows from the generalization of the Paley-Wiener theorem established by Stein (see~\cite[\S III, Theorem 4]{stein1957}
or~\cite[\S III Theorem 4.9]{stein1971}), that
\begin{equation}\label{intro166}
\wF(\bt) = \int_{\R^N} F(\bx) e\bigl(-\bt^T \bx\bigr)\ \dbx = 0,\quad\text{if $\bt \in \R^N$ and $\delta\le |\bt|$}.
\end{equation}

The hypothesis (\ref{intro166}) leads to a~simple form of the Poisson summation formula. Because the $N \times N$ matrix $S$ is symmetric the Poisson
summation formula takes the form (see~\cite[Chapter VII, Corollary 2.6]{stein1971})
\begin{equation}\label{intro174}
\begin{split}
(\det S) \sum_{\bm\in\Z^N}F\bigl(S(\bm + \bx)\bigr) &= \sum_{\bn\in\Z^N}\wF\bigl(S^{-1}\bn\bigr)e\bigl((S\bx)^T S^{-1}\bn\bigr)\\
	&= \sum_{\bn\in\Z^N}\wF\bigl(S^{-1}\bn\bigr)e\bigl(\bx^T S^T S^{-1}\bn\bigr)\\
	&= \sum_{\bn\in\Z^N}\wF\bigl(S^{-1}\bn\bigr)e\bigl(\bx^T \bn\bigr).
\end{split}
\end{equation}
And because $\wF$ is a~continuous function with compact support, it follows that each sum on the right of (\ref{intro174}) contains only finitely many
nonzero terms. Applying (\ref{intro166}) we find that
\begin{equation}\label{intro179}
(\det S) \sum_{\bm\in\Z^N}F\bigl(S(\bm + \bx)\bigr)
	= \sum_{\substack{\bn\in\Z^N \\|S^{-1}\bn| < \delta}}\wF\bigl(S^{-1}\bn\bigr)e\bigl(\bx^T \bn\bigr).
\end{equation}
The function on the right of (\ref{intro179}) is a~finite linear combination of continuous characters
\begin{equation*}\label{intro183}
\bx \mapsto e\bigl(\bx^T \bn\bigr)
\end{equation*}
defined on the compact group $\R^N /\Z^N$. That is, the function of $\bx$ on the right of (\ref{intro179}) is a~trigonometric polynomial.
If we assume that $\delta \le |A|^{-1}$, then it follows from (\ref{intro108}) in the statement of Lemma~\ref{lemappl1} that there
is at most one nonzero term in the sum on the right of (\ref{intro179}). That is, if we assume that $\delta \le |A|^{-1}$ then (\ref{intro179}) simplifies to
\begin{equation}\label{intro186}
(\det S) \sum_{\bm\in\Z^N}F\bigl(S(\bm + \bx)\bigr) = \wF(\bo)
\end{equation}
at each coset representative $\bx$ in $\R^N /\Z^N$.
\section{Proof of Theorem~\ref{thmintro1}}

In this section we prove Theorem~\ref{thmintro1} by applying the Poisson formula (\ref{intro186}) to the two real entire functions
\begin{equation*}\label{intro279}
\F_{\nu}(\bz; R, \delta) = \psi_N\bigl(f_{\nu}\bigr)(\bz; R, \delta),\quad\text{and}\quad \G_{\nu}(\bz; R, \delta) = \psi_N\bigl(g_{\nu}\bigr)(\bz; R, \delta)
\end{equation*}
which were initially defined in (\ref{intro238}). We recall that $f_{\nu}$ and $g_{\nu}$ are even entire functions of exponential type at most $2 \pi \delta$
defined in (\ref{intro218}) and (\ref{intro222}). The special functions $\F_{\nu}$ and $\G_{\nu}$ satisfy the basic inequality
\begin{equation*}\label{intro280}
\F_{\nu}(\bx; R, \delta) \le \chi_R(\bx) \le \G_{\nu}(\bx; R, \delta)
\end{equation*}
at each point $\bx$ in $\R^N$.

We recall that $2 \nu + 2 = N$. It follows from Lemma~\ref{lemboas3} that $\F_{\nu}$ restricted to $\R^N$ and $\G_{\nu}$ restricted
to $\R^N$ belong to
\begin{equation*}\label{intro281}
L^1 \bigl(\R^N \bigr) \cap L^2 \bigl(\R^N \bigr).
\end{equation*}
It also follows from Lemma~\ref{lemboas3} that their Fourier transforms
\begin{equation*}\label{intro282}
\wsF_{\nu} : \R^N \rightarrow \C,\quad\text{and}\quad \wsG_{\nu} : \R^N \rightarrow \C,
\end{equation*}
are continuous functions supported on the closed ball
\begin{equation*}\label{intro283}
\big\{\bt \in \R^N : |\bt| \le \delta\big\}.
\end{equation*}

Using (\ref{intro249}) and (\ref{intro186}) we find that
\begin{equation}\label{intro288}
\begin{split}
\wsF_{\nu}(\bo; R,\delta) &= (\det S) \sum_{\bm\in\Z^N}\F_{\nu}\bigl(S(\bm + \bx); R,\delta\bigr)\\
	&\le (\det S) \sum_{\bm\in\Z^N} \chi_R\bigl(S(\bm + \bx)\bigr)\\
	&\le (\det S) \sum_{\bm\in\Z^N}\G_{\nu}\bigl(S(\bm + \bx); R,\delta\bigr)\\
	&= \wsG_{\nu}(\bo; R,\delta).
\end{split}
\end{equation}
We select the real parameter $\nu$ so that $2\nu + 2 = N$. Then it follows from (\ref{intro253}) that
\begin{equation}\label{intro290}
 \wsG_{\nu}(\bo; R,\delta) - \wsF_{\nu}(\bo; R,\delta) = \omega_{N-1}u_{\nu}(R,\delta).
\end{equation}
Also, using (\ref{intro249}) again we get
\begin{equation}\label{intro297}
\wsF_{\nu}(\bo; R,\delta) \le \int_{\R^N} \chi_R(\bx)\ d\bx = V_N R^N \le \wsG_{\nu}(\bo; R,\delta),
\end{equation}
where $V_N$ is given by (\ref{intro44}). Now (\ref{intro288}), (\ref{intro290}), and (\ref{intro297}), imply that the inequality
\begin{equation}\label{intro304}
\begin{split}
V_N R^N - \omega_{N-1}u_{\nu}(R,\delta) &\le \wsF_{\nu}(\bo; R,\delta)\\
		&\le (\det S) \sum_{\bm\in\Z^N} \chi_R\bigl(S(\bm + \bx)\bigr)\\
		&\le \wsG_{\nu}(\bo; R,\delta) \le V_N R^N + \omega_{N-1}u_{\nu}(R,\delta)
\end{split}
\end{equation}
holds at each point $\bx$ in $\R^N$. Therefore (\ref{intro304}) can be rewritten as the inequality
\begin{equation}\label{intro311}
\begin{split}
\biggl|(\det S)\sum_{\bm\in\Z^N} \chi_R\bigl(S(\bm + \bx)\bigr) - V_N R^N \biggr| \le \omega_{N-1}u_{\nu}(R,\delta)
\end{split}
\end{equation}
for all points $\bx$ in $\R^N$. Finally, we can express (\ref{intro311}) in terms of the original matrix $A$ and the problem of estimating (\ref{intro68})
by applying the identity (\ref{intro154}). This completes the proof of Theorem~\ref{thmintro1}.

{\small\bibliography{lattice}}

\EditInfo{March 28, 2023}{April 26, 2023}{Camilla Hollanti and Lenny Fukshansky}

\end{document}